\documentclass[a4paper]{amsart}
\usepackage[english]{babel}
\usepackage[a4paper,top=3cm,bottom=2cm,left=3cm,right=3cm,marginparwidth=1.75cm]{geometry}
\usepackage{amsmath}
\numberwithin{equation}{section}

\DeclareMathOperator{\Tr}{Tr}
\DeclareMathOperator{\tr}{tr}\usepackage{graphicx}
\usepackage[colorinlistoftodos]{todonotes}
\usepackage[colorlinks=true, allcolors=blue]{hyperref}
\usepackage{graphicx}
\usepackage{lmodern}
\usepackage[english]{babel}
\usepackage{amsmath}
\usepackage{color}
\usepackage{amsfonts}
\usepackage{amssymb}
\usepackage{amsthm}
\usepackage{stmaryrd}

\usepackage{pstricks}
\usepackage{epsfig}
\newtheorem{teo}{Theorem}[section]
\newtheorem{pro}[teo]{Proposition}

\newtheorem{cor}[teo]{Corollary}
\newtheorem{rem}[teo]{Remark}
\title[Moments]{On star-Moments of the compression of the free unitary Brownian motion by a free projection}
\author[N. Demni]{Nizar Demni}
\address{ Aix-Marseille Universit\'e \\ CNRS\\  Centrale Marseille \\ I2M - UMR 7373\\  39 rue F. Joliot Curie\\  13453 Marseille \\ France }
\email{nizar.demni@univ-amu.fr}

\author[T. Hamdi]{Tarek Hamdi}
\address{Department of Management Information Systems \\ College of Business Management \\ Qassim University \\ Ar Rass \\ Saudi Arabia
and Laboratoire d'Analyse Math\'ematiques et Applications LR11ES11 \\ Universit\'e de Tunis El-Manar \\ Tunisie}
\email{t.hamdi@qu.edu.sa } 
\usepackage{graphicx}

\begin{document}

\maketitle

\begin{abstract}
In this paper, we derive explicit expressions for the moments and for the mixed moments of the compression of a free unitary Brownian motion by a free projection. While the moments of this non-normal operator are readily derived using analytical or combinatorial methods, we only succeeded to derive its mixed ones after solving a non-linear partial differential equation for their two-variables generating function. Nonetheless, the combinatorics of non-crossing partitions  lead to another expression of the lowest-order mixed moment. We shall also give some interest in odd alternating moments. In particular, we derive a linear partial differential equation for their generating function which we solve explicitly when the rank of the projection equals $1/2$. As to the combinatorial approach, it leads in this case to the analysis of Kreweras complements of particular non-crossing partitions. For general ranks, the analytic approach rather provides an implicit solution. 
\end{abstract}


\section{Introduction}

Let $(\mathcal{A}, \tau)$ be a non-commutative $\star$-probability space: $\mathcal{A}$ is a unital von Neumann algebra equipped with a faithful tracial state $\tau$. If $a \in \mathcal{A}$ is a non-normal operator, then one has:
\begin{equation*}
\tau(a^m) = \int z^m \mu_a(dz),  \quad \tau((a^{\star})^m) = \int \overline{z}^m \mu_a(dz),
\end{equation*}
for all $m \geq 0$, where $\mu_a$ is the Brown measure of $a$ (see e.g. \cite{Min-Spe}). However, the equality for 
\begin{equation*}
\tau((a^m)(a^{\star})^n)) = \int z^m \overline{z}^n\mu_a(dz), \quad m, n \geq 0, 
\end{equation*}  
does not necessarily hold unless $a$ is a normal operator. As a matter of fact, computing $\tau((a^m)(a^{\star})^n))$ when $a$ is non normal and more generally 
\begin{equation*}
\tau(p(a,a^{\star})),
\end{equation*}
where $p$ ranges over the space of all polynomials in two non-commuting indeterminates, may turn out to be quite difficult and even to be out of reach. 


Let $P$ be a self-adjoint projection and let $(Y_t)_{t \geq 0}$ be a free unitary Brownian motion, both operators being free (in Voiculescu's sense) in $\mathcal{A}$. Then the compression $PY_tP$ is in general a non-normal operator. 
The following moments: 
\begin{equation*}
\tau((PY_tPY_t^{\star}P)^n), \quad n \geq 0,
\end{equation*}   
in the compressed algebra were already computed in \cite{Dem-Ham} in relation to the free Jacobi process and correspond in the non normal picture to the polynomial:
\begin{equation*}
p(a,a^{\star}) = (aa^{\star})^n, \quad a = PY_tP.
\end{equation*}
We shall refer to them as even alternating moments of $PY_tP$ since the word $(aa^{\star})^n$ is alternating in the alphabet $(a, a^{\star})$ and has $2n$ elements. However, up to our best knowledge, no other $\star$-moments of $PY_tP$ were considered or computed in literature. 

In this paper, we are interested in computing the moments, the mixed moments and the odd alternating moments of $PY_tP$. The two latter moments correspond respectively to polynomials of the form: 
\begin{equation*}
p(a,a^{\star}) = a^m(a^{\star})^n, \quad m,n \geq 1, \quad p(a,a^{\star}) = (aa^{\star})^ma, \quad m \geq 1.
\end{equation*}
This choice is motivated by our guess confirmed in the sequel that the mixed and the odd alternating moments are still attainable though exhibit quite complicated analytic and combinatorial structures. Of course, computing the moments 
\begin{equation*}
\tau((PY_tP)^n) = \tau((PY_t)^n), \quad n \geq 0,
\end{equation*}
is straightforward and we perform it here for sake of completeness. In this respect, we shall prove in three different ways that they coincide, up to dilations, with those of $Y_t$. The first way relies on an application of the free It\^o's formula leading to a partial differential equation (hereafter pde) satisfied by their generating function. Another one appeals to the moment formula for alternating products of free random variables (\cite{Nic-Spe}) together with the free cumulants of $Y_t$ (\cite{Levy}). We shall also determine them as large-size limits of their finite-dimensional counterparts whose expressions follow from T. L\'evy's expansion for observables of the Brownian motion on the group of unitary matrices (\cite{Levy}).  

Coming to the computations of the mixed moments: 
\begin{equation*}
\tau[(PYP)^m(PY^{\star}P)^n] = \tau[(PY)^m(PY^{\star})^n], \quad m,n \geq 1,
\end{equation*}
 we shall use again the free stochastic calculus to derive a pde for their two-variables generating function. Its unique solution (in the space of two-variables analytic functions around $(0,0)$) is then expressed through the moment generating functions of $\tau((PY_t)^n)$ in each variable. Extracting the Taylor coefficients of this solution, the mixed moments are then expressed as linear combinations of product of Laguerre polynomials. Once this expression is obtained, we attempt to reprove it relying on the combinatorics of non crossing partitions and on the knowledge of the mixed cumulants of $(Y_t, Y_t^{\star})$ derived in \cite{DGN}. For instance, the lowest-order mixed moment corresponding to $m \geq 1, n =1,$ may be expressed as a weighted sum of Laguerre polynomials. However, the combinatorial approach becomes rapidly quite involved for higher orders and does not lead to a tractable formula compared with the one obtained using the analytic approach. 

As far as the odd alternating moments of $PY_tP$ are concerned, we also derive a linear pde for their generating function. It turns out that this pde is linear and that its coefficients involve the moment generating function of the free Jacobi process. The last fact is somehow expected since even alternating moments of $PY_tP$ are nothing else but the moments of the free Jacobi process. Using the method of characteristics, we shall solve explicitly the obtained pde in the particular case $\tau(P) = 1/2$. Actually, our analysis reveals in this case that the characteristics are inherited from those governing the dynamics of the spectral distribution of $Y_t$. Moreover, the squared solution admits a product form which does not seem to exist for general ranks $\tau(P) \neq 1/2$. 
More generally, the method of characteristics shows that the solution of this pde may be expressed through the Herglotz transform of some unitary operator, the latter being however implicit. 

We close the paper by discussing the combinatorial approach to these (odd) alternating moments. If $\tau(P) = 1/2$, we are then led to the analysis of the block structure of the Kreweras complement of special partitions. Indeed, all odd free cumulants of $P$ vanish in this case except the first one, so that we are left with non crossing partitions formed by an odd number of singletons and/or even non crossing partitions. However, the block structure of the Kreweras complement of such paritition depends on the position of the singletons there which makes the combinatorial approach very likely less efficient than the analytical one.

The paper is organised as follows. The next section is devoted to the moments of $PY_tP$ while the third one is devoted to its mixed moments. The odd alternating moments are studied in the last section. 
\section{Moments of $PY_t$}
Let $(Y_t)_{t \geq 0}$ be a free unitary Brownian motion in $\mathcal{A}$. This is a free It\^o process solution of the following free stochastic differential equation (\cite{Biane}): 
\begin{equation}\label{SDEFUBM}
dY_{t} = iY_tdX_t-\frac{1}{2}Y_tdt, 
\end{equation}
where $(X_t)_{t \geq 0} \in \mathcal{A} $ is a free additive Brownian motion. It is also the large-size limit of the Brownian motion on the unitary group and its moments are given by (\cite{Biane}): 
\begin{equation*}
\tau(Y_t^n) = \frac{e^{-nt/2}}{n}L_{n-1}^{(1)}(nt) = \frac{e^{-nt/2}}{n}\sum_{r=0}^{n-1}\binom{n}{r+1}\frac{(-nt)^r}{r!} , \quad n \geq 1,
\end{equation*}
where $L_{n-1}^{(1)}$ is the $(n-1)$-th Laguerre polynomial of parameter one. Now, let $P = P^{\star} \in \mathcal{A}$ be a self-adjoint projection with continuous rank $\tau(P) = \alpha \in (0,1]$ and assume that $P$ and $(Y_t)_{t \geq 0}$ are $\star$-free in Voiculescu's sense in $\mathcal{A}$. In this section, we shall write three different proofs of the following result: 
\begin{pro}
For all $n \geq 0, t \geq 0$,
\begin{equation*}
\tau\big[(PY_t )^n \big]=\alpha e^{-n\frac{t}{2}}Q_n(\alpha t), 
\end{equation*}
where $Q_0(t) \equiv 1$ and  $Q_n(t) :=  e^{nt/2}\tau(Y_t^n), n \geq 1$.
\end{pro}
\begin{proof}
{\it First proof}: Using \eqref{SDEFUBM}, the process
\begin{equation*}
A_t := e^{t/2}PY_t, 
\end{equation*}
satisfies the free stochastic differential equation:
\begin{equation*}
dA_t = iA_tdX_t.
\end{equation*}
Now, a straightforward application of the free It\^o formula yields (see e.g. \cite{De}, Proposition 6.1): 
\begin{align*}
 dA_t^n=i\sum_{k=1}^nA_t^kdX_tA_t^{n-k}-\sum_{k=1}^{n-1}kA_t^k\tau(A_t^{n-k})dt, \quad n \geq 1.
 \end{align*}
Taking the expectation with respect to $\tau$ in both sides of this last equality, we obtain the following ordinary differential equation (ode): 
\begin{align*}
 \frac{d}{dt}\tau(A_t^n)=-\sum_{k=1}^{n-1}k\tau(A_t^k)\tau(A_t^{n-k}).
 \end{align*}
Equivalently, the moment generating function 
\begin{equation*}
\psi(t,z) := \sum_{n\geq1}\tau(A_t^n)z^n
\end{equation*}
satisfies the pde:
\begin{equation*}
\partial_t\psi+z\psi\partial_z\psi=0,\quad \psi(0,z)=\frac{\alpha z}{1-z}.
\end{equation*}
But, it is known from \cite{Biane} that:
\begin{equation*}
\eta(t,z) :=\sum_{n\geq1}Q_n(t)z^n
\end{equation*}
is the unique solution of
\begin{equation}\label{PDEUBM}
\partial_t\eta+z\eta\partial_z\eta=0,\quad \eta(0,z)=\frac{ z}{1-z}, 
\end{equation}
which is analytic in the open unit disc. Consequently, $\psi(t,z)=\alpha\eta(\alpha t,z)$ and the first proof is completed. 

{\it Second proof}: 
For $n \geq 1$, let $k_n(Y_t)$ be the $n$-th free cumulant of $Y_t$ and for a given non-crossing partition $\pi$ with blocks $V_1, \dots, V_r$, recall the cumulant and the moment multiplicative functionals (\cite{Nic-Spe}, Lecture 11): 
\begin{eqnarray*}
k_{\pi}[Y_t,...,Y_t] &:=&   \prod_{i=1}^rk_{|V_i|}(Y_t), \\ 
\tau_{K(\pi)}[P,...,P]  &:=&   \prod_{i=1}^r\tau\left(P^{|V_i|}\right).
 \end{eqnarray*}
Then, the moment formula for alternating products of free random variables reads (\cite{Nic-Spe}, Theorem 14.4):
\begin{align*}
\tau\big[(PY_t )^n \big]=\sum_{\pi\in NC(n)} k_{\pi}[Y_t,...,Y_t]\tau_{K(\pi)}[P,...,P],
\end{align*}
where $K(\pi)$ is the Kreweras complement of $\pi$ (\cite{Nic-Spe}, Definition 9.21). Since $P^k=P$ for all $k\geq1$ and since $|K(\pi)|=n+1-|\pi|$, then 
\begin{align*}
\tau_{K(\pi)}[P,...,P] = \alpha^{|K(\pi)|} = \alpha^{n+1 - |\pi|}.
\end{align*}
Furthermore, the free cumulants of $Y_t$ are given by (\cite{Levy}):
\begin{align*}
k_{n}(Y_t)=e^{-nt/2}\frac{(-n)^{n-1}}{n!}t^{n-1},\quad n \geq 1, t \geq 0.
\end{align*}
As a result,  
\begin{align*}
\tau\big[(PY_t )^n \big]&=\sum_{r=1}^n\sum_{\pi=(V_1,...,V_r)\in NC(n)} \prod_{i=1}^rk_{|V_i|}(Y_t)\alpha^{n+1-r}
\\&=\sum_{r=1}^n\alpha^{n+1-r}\sum_{\pi=(V_1,...,V_r)\in NC(n)} \prod_{i=1}^r e^{-|V_i|\frac{t}{2}}\frac{(-|V_i|)^{|V_i|-1}}{|V_i|!}t^{|V_i|-1}
\\&=\alpha e^{-n\frac{t}{2}}\sum_{r=1}^n(-\alpha t)^{n-r}\sum_{\pi=(V_1,...,V_r)\in NC(n)} \prod_{i=1}^r \frac{|V_i|^{|V_i|-1}}{|V_i|!}
\\&=\alpha e^{-nt/2}\sum_{r=0}^{n-1}(-\alpha t)^{r}\sum_{\pi=(V_1,...,V_{n-r})\in NC(n)} \prod_{i=1}^{n-r} \frac{|V_i|^{|V_i|-1}}{|V_i|!}.
\end{align*}
Hence, the second proof is complete provided that the following formula holds:
\begin{equation}\label{binom}
\sum_{\pi=(V_1,...,V_{n-r})\in NC(n)} \prod_{i=1}^{n-r} \frac{|V_i|^{|V_i|-1}}{|V_i|!}=\frac{n^{r-1}}{r!}\binom{n}{r+1}, \quad n\geq1,r\in\{0,...n-1\}. 
\end{equation}
But the latter follows immediately by equating the coefficients of $t$ in the moment-cumulant formula:
\begin{align*}
Q_n(t)=e^{nt/2}\tau\big[Y_t ^n \big]=\sum_{r=0}^{n-1}(- t)^{r}\sum_{\pi=(V_1,...,V_{n-r})\in NC(n)} \prod_{i=1}^{n-r} \frac{|V_i|^{|V_i|-1}}{|V_i|!}.
\end{align*}

{\it Third proof}: Let $(Y_t^N)_{t\geq0}$ be a standard Brownian motion on the unitary group $\mathcal{U}_N$ and let $P^N$ be a $N\times N$ orthogonal projection such that: 
\begin{equation*}
\lim_{N \rightarrow \infty} \tr_N(P^N) = \alpha, 
\end{equation*}
where $\tr_N = (1/N)\Tr$ is the normalized trace on the space of complex $N \times N$ matrices. Then, $P^N$ and $(Y_{t/N}^N)_{t\geq0}$ are asymptotically free as $N \rightarrow \infty$ and converge respectively (in the sense of noncommutative moments) to $P$ and $(Y_t)_{t \geq 0}$ (see \cite{Biane}). 
 
Now, let $\sigma$ be an element of the symmetric group $\mathfrak{S}_n$ and set:
\begin{equation*}
p_{\sigma}(P^NY_{\frac{t}{N}}^N ) :=\prod_{c=(i_1,...,i_r)\ {\rm cycle\ of}\ \sigma}\tr_N\left[(P^NY_{\frac{t}{N}}^N )^r\right].
\end{equation*}
Then, the expansion proved in \cite[Theorem 3.3]{Levy} reads:
\begin{equation}\label{mom}
\mathbb{E}\big[p_{\sigma}(P^NY_{\frac{t}{N}}^N ) \big]=e^{-nt/2}\sum_{k,d=0}^{\infty}\frac{(-1)^kt^k}{k!N^{2d}}\tr_N\left[(P^N)^{n-|\sigma|+k-2d}\right]S(\sigma,k,d),
\end{equation}
where $S(\sigma,k,d)$ is the number of paths starting at $\sigma$ in the Cayley graph of $\mathfrak{S}_n$, of length $k$ and defect $d$, and $|\sigma|$ is the length of $\sigma$ with respect to the set of transpositions (see \cite{Levy} for further details).
In particular, if $\sigma = (1\dots n)$ then, 
\begin{equation*}
\mathbb{E}\big[\mathop{tr_N} (P^NY_{\frac{t}{N}}^N )^n \big]= e^{-n\frac{t}{2}}\sum_{k,d=0}^{\infty}\frac{(-1)^kt^k}{k!N^{2d}}\tr_N\left[(P^N)^{k+1-2d}\right]S((1...n),k,d), 
\end{equation*}
which together with Proposition 6.6. in \cite{Levy} imply:
\begin{align*}
\lim_{N\rightarrow\infty}\mathbb{E}\big[\mathop{tr_N}(P^NY_{\frac{t}{N}}^N )^n \big] & = e^{-n\frac{t}{2}}\sum_{k=0}^{n-1}\frac{(-1)^kt^k}{k!}\alpha^{k+1}S((1...n),k,0) 
\\& = e^{-n\frac{t}{2}}\sum_{k=0}^{n-1}\frac{(-1)^kt^k}{k!}\alpha^{k+1} \binom{n}{k+1}n^{k-1}= \alpha Q_n(\alpha t). 
\end{align*}
\end{proof}

\section{Mixed moments}
Now, we proceed to the computation of the mixed moments: 
\begin{equation*}
R^{(\alpha)}_{m,n}(t)=e^{\frac{t(m+n)}{2}}\tau[(PY_t)^m(PY_t^{\star})^n], \quad m, n \geq 0, 
\end{equation*}
where we set $R^{(\alpha)}_{0,0}(t):= \tau(P) = \alpha$\footnote{This choice of the initial value ensures that $R^{(\alpha)}_{m,n}(0) = \alpha$ for all $(m,n)$.}. Obviously, $R^{(\alpha)}_{m,n}(t)$ has the following properties:
\begin{itemize}
    \item $R^{(\alpha)}_{m,n}(t)=R^{(\alpha)}_{n,m}(t)$ for all $m,n\in\mathbb{N}$.
    \item $R^{(\alpha)}_{m,0}(t)=\alpha Q_m(\alpha t)$ for all $m\in\mathbb{N}^*$.
    \item $R^{(1)}_{m,n}(t)= Q_{|m-n|}( t)$.
     \item $R^{(\alpha)}_{m,n}(0)=\alpha$.
\end{itemize}
In particular, 
\begin{align*}
R^{(\alpha)}_{1,0}(t)=e^{t/2}\tau(PY_t)=e^{t/2}\tau(P)\tau(Y_t)=\alpha,
\end{align*}
and using the freeness definition, we readily derive: 
\begin{align*}
R^{(\alpha)}_{1,1}(t)=e^{t}\tau(PY_tPY_t^{\star})=\alpha^2 e^{t}+\alpha(1-\alpha).
\end{align*}
More generally, $t \mapsto R^{(\alpha)}_{m,n}(t)$ satisfies the following ode.
\begin{pro}\label{ode}
 For all $n,m\geq1$ and any $t > 0$, we have
\begin{multline*}
\frac{d}{dt}R^{(\alpha)}_{m,n}(t) = -\alpha\sum_{k=1}^{n-1}kR^{(\alpha)}_{m,k}(t)Q_{n-k}(\alpha t)-\alpha\sum_{k=1}^{m-1}kR^{(\alpha)}_{k,n}(t)Q_{m-k}(\alpha t)
\\ +e^t\sum_{ k=0}^{ m-1}\sum_{j=0}^{n-1} R^{(\alpha)}_{k,j}(t) R^{(\alpha)}_{m-1-k,n-1-j}(t),
\end{multline*}
where an empty sum is zero. 
 \end{pro}
\begin{proof}
Recall the notation $A_t := e^{t/2}PY_t$ and introduce similarly $B_t := e^{t/2}PY_t^{\star}$ so that $R^{(\alpha)}_{m,n}(t) = \tau(A_t^mB_t^n)$. Then,
\begin{equation*}
d\big[A_t^m B_t^n \big]=d\big[A_t^m \big] B_t^n+A_t^m d\big[B_t^n \big]+\big(dA_t^m\big)\big(dB_t^n \big),
\end{equation*}
where $\big(dA_t^m\big)\big(dB_t^n \big)$ is the bracket of the free semi-martingales $A^m$ and $B^m$ at time $t$, and 
\begin{eqnarray*}
dA_t^m & = & i\sum_{k=1}^mA_t^kdX_tA_t^{m-k}-\sum_{k=1}^{m-1}kA_t^k\tau(A_t^{m-k})dt,  \\ 
dB_t^n & = & -ie^{t/2}\sum_{k=1}^nB_t^{n-k}PdX_tY_t^{\star}B_t^{k-1}-\sum_{k=1}^{n-1}kB_t^k\tau(B_t^{n-k})dt.  
\end{eqnarray*}
As a result, one gets:
\begin{align*}
d\big[A_t^m B_t^n \big]=&i\sum_{k=1}^mA_t^kdX_tA_t^{m-k}B_t^n-\sum_{k=1}^{m-1}kA_t^kB_t^n\tau(A_t^{m-k})dt
\\&-ie^{t/2}\sum_{k=1}^nA_t^mB_t^{n-k}PdX_tY_t^{\star}B_t^{k-1}-\sum_{k=1}^{n-1}kA_t^mB_t^k\tau(B_t^{n-k})dt
\\&+e^{t/2}\left( \sum_{k=1}^nA_t^kdX_tA_t^{m-k}\right)\left(\sum_{k=1}^nB_t^{n-k}PdX_tY_t^{\star}B_t^{k-1}\right).
\end{align*} 
But, the last semi-martingale bracket is explicitly computed as: 
\begin{align*}
  e^{t/2}\sum_{ k=1}^{m}\sum_{j=1}^{n}A_t^kY_t^{\star}B_t^{j-1}\tau\big(A_t^{m-k} B_t^{n-j}P \big)dt&= e^{t}\sum_{ k=1}^{ m}\sum_{j=1}^{n}A_t^{k-1}PB_t^{j-1}\tau\big(A_t^{m-k} B_t^{n-j}P \big)dt
  \\&= e^{t}\sum_{ k=0}^{ m-1}\sum_{j=0}^{n-1}A_t^kB_t^{j}\tau\big(A_t^{m-1-k} B_t^{n-1-j} \big)dt.
\end{align*}
Consequently
\begin{align*}
\frac{d}{dt} R^{(\alpha)}_{m,n}(t) = \frac{d}{dt}\tau\big(A_t^n B_t^m \big)=&-\sum_{k=1}^{m-1}k\tau(A_t^kB_t^n)\tau(A_t^{m-k})-\sum_{k=1}^{n-1}k\tau(A_t^mB_t^k)\tau(B_t^{n-k})
\\&+e^t\sum_{ k=0}^{ m-1}\sum_{j=0}^{n-1}\tau(A_t^{k}B_t^{j})\tau\big(A_t^{m-1-k} B_t^{n-1-j} \big).
\end{align*}
Finally, since $\tau(A_t^{m-k}) = \alpha Q_{m-k}(\alpha t)$ and likewise $\tau(B_t^{n-k}) = \alpha Q_{n-k}(\alpha t)$, the proposition follows. 
\end{proof}
Now, we turn the previous ode to a pde for the following moment generating function:
 \begin{equation*}
     \phi^{(\alpha)}(t,y,z)=\sum_{m,n\geq0}R^{(\alpha)}_{m,n}(t)y^mz^n= \alpha + \alpha \eta(\alpha t,y)+\alpha \eta(\alpha t,z)+\sum_{m,n\geq1}R^{(\alpha)}_{m,n}(t)y^mz^n,
 \end{equation*}
 which converges absolutely for any fixed time $t \geq 0$ in some neighborhood of $(0,0)$ and satisfies $\phi^{(\alpha)}(t,0,0) = \alpha$. More precisely, one has
\begin{cor}
 For all $\alpha\in(0,1]$,
 \begin{equation*}
\partial_t\phi^{(\alpha)}+\alpha y\eta(\alpha t,y)\partial_y\phi^{(\alpha)}+\alpha z\eta(\alpha t,z)\partial_z\phi^{(\alpha)}=e^tyz\phi^{(\alpha)}\big(t,y, z\big)^2.
 \end{equation*}
where we recall that 
\begin{equation*}
\eta(t,z) :=\sum_{n\geq1}Q_n(t)z^n
\end{equation*}
\end{cor}
Note that if $y=0$, then we recover the pde satisfied by:  
\begin{equation*}
\phi^{(\alpha)}(t,z,0) = \alpha(1+\eta(\alpha t,z)).
\end{equation*}
Moreover, the function
\begin{equation*}
f^{(\alpha)}(t,y,z) :=\frac{-1}{\phi^{(\alpha)}(t,y,z)},
\end{equation*}
is well defined for any fixed time $t > 0$ in a neighborhood of $(0,0)$ and satisfies the linear pde:
\begin{equation}\label{ode}
\partial_tf^{(\alpha)}+\alpha y\eta(\alpha t,y)\partial_yf^{(\alpha)}+\alpha z\eta(\alpha t,z)\partial_zf^{(\alpha)}=e^tyz.
 \end{equation}
Since $\eta(t,\cdot)$ is a bi-holomorphic map from the an open disc centered at the origin onto some Jordan domain with inverse given there by (\cite{Biane1}): 
\begin{equation*}
\eta^{-1}( t,y)=\frac{y}{1+y}e^{ t y},
\end{equation*}
then we can set 
\begin{equation*}
g^{(\alpha)}(t,y,z) := f^{(\alpha)}(t,\frac{ye^{\alpha t y}}{1+y},\frac{ze^{\alpha t z}}{1+z})
\end{equation*}
or equivalently: 
\begin{equation*}
f^{(\alpha)}(t,y,z)=g^{(\alpha)}(t,\eta(\alpha t,y),\eta(\alpha t,z)).
\end{equation*}
In this respect, we shall prove that: 
\begin{pro}
The function $g^{(\alpha)}$ admits the following expression:
\begin{equation*}
g^{(\alpha)}(t,y,z)=\frac{yz(e^{t(1+\alpha y+\alpha z)}-1)}{(1+\alpha y+\alpha z)(1+y)(1+z)}-\frac{1}{\alpha(1+y)(1+z)}.
\end{equation*}
\end{pro}

\begin{proof}
We readily compute:
\begin{align*}
\partial_t f^{(\alpha)}(t,y,z)&=\partial_t g^{(\alpha)}+\alpha \partial_t \eta\partial_y g^{(\alpha)}+\alpha \partial_t \eta\partial_z g^{(\alpha)},
\\ \partial_y f^{(\alpha)}(t,y,z)&=\partial_y \eta\partial_y g^{(\alpha)},
\\\partial_z f^{(\alpha)}(t,y,z)&=\partial_z \eta\partial_z g^{(\alpha)}.
\end{align*}
Substituting these partial derivatives in \eqref{ode} and using the fact that $\partial_t \eta=-y\eta\partial_y \eta$, we get the following ode:
\begin{equation*}
     \partial_tg^{(\alpha)}(t,\eta(\alpha t,y),\eta(\alpha t,z) )=e^tyz.
 \end{equation*}
 in some neighborhood of $(0,0)$. Equivalently,
\begin{equation*}
     \partial_tg^{(\alpha)}(t,y,z )=\frac{yz}{(1+y)(1+z)}e^{t(1+\alpha y+\alpha z)},
 \end{equation*}
 near $(0,0)$. Integrating the last ode and taking into account the initial value at $t=0$:
\begin{equation*}
g^{(\alpha)}(0,y,z)=f^{(\alpha)}(0,\frac{y}{1+y},\frac{z}{1+z})=\frac{-1}{\alpha(1+y)(1+z)},
\end{equation*}
we get the sought expression. 
\end{proof}

With these findings, we can now state and prove the main result of this section: 
\begin{teo}
Define $(c_{j,k}(t, \alpha))_{j, k \geq 0}$ by: 
\begin{equation*}
-\frac{1}{g^{(\alpha)}(t,y,z)} =  \sum_{j,k \geq 0} c_{j,k}(t, \alpha) y^j z^k, 
\end{equation*}
in a neighborhood of $(0,0)$. Then for any $m,n \geq 1$, 
\begin{equation*}
R_{m,n}(t) = \frac{1}{m n}\left\{\sum_{j=1}^m\sum_{k=1}^n (jk) c_{j,k}(t, \alpha)L_{m-j}^{(j)} (m \alpha t)L_{n-k}^{(k)} (n \alpha t)\right\}. 
\end{equation*}
where $L_{n-k}^{(k)}$ is the $(n-k)$-th Laguerre polynomial of parameter $k$. 
\end{teo}
\begin{proof}
Keeping in mind the relations between $g^{(\alpha)}, f^{(\alpha)},$ and $\phi^{(\alpha)}$, it follows that for any time $t > 0$, 
\begin{equation*}
\phi^{(\alpha)}(t,y,z) =  \alpha +\alpha \eta(\alpha t, y) + \alpha \eta(\alpha t, z) + \sum_{j,k \geq 1} c_{j,k}(t, \alpha) [\eta(\alpha t, y)]^j [\eta(\alpha t, z)]^k, 
\end{equation*}
near $(0,0)$. But, for any $j, k \geq 1$, the following identity holds (\cite{Dem}):
\begin{equation*}
[\eta(\alpha t, y)]^j = j \sum_{m \geq j} L_{m-j}^{(j)} (m \alpha t)\frac{y^m}{m},
\end{equation*}
and likewise 
\begin{equation*}
[\eta(\alpha t, z)]^k = k \sum_{n \geq k} L_{n-k}^{(k)} (n \alpha t)\frac{z^n}{n}. 
\end{equation*}
Consequently, 
\begin{multline*}
\phi^{(\alpha)}(t,y,z) =  \alpha+\alpha \eta(\alpha t, y) + \alpha \eta(\alpha t, z)  \\ + \sum_{m,n \geq 1} \left\{\sum_{j=1}^m\sum_{k=1}^n (jk) c_{j,k}(t, \alpha)L_{m-j}^{(j)} (m \alpha t)L_{n-k}^{(k)} (n \alpha t)\right\} \frac{y^m}{m}\frac{z^n}{n}, 
\end{multline*}
whence the corollary follows. 
\end{proof}

\begin{rem}
We can compute explicitly the Taylor coefficients $c_{j,k}(t, \alpha)$ using the geometric series and the binomial Theorem. Indeed, for small enough $|y|, |z|$, we expand: 
\begin{align*}
-\frac{1}{g^{(\alpha)}(t,y,z)} & = \frac{1+\alpha y+\alpha z+ \alpha yz}{1-\displaystyle \frac{yz(1+\alpha y + \alpha z + \alpha yz)[e^{t(1+\alpha y+\alpha z)}-1]}{(1+\alpha y + \alpha z)(1+y)(1+z)}} 
\\& = (1+\alpha y+\alpha z+ \alpha yz)\sum_{q \geq 0} \frac{(yz)^q}{[(1+y)(1+z)]^q}[e^{t(1+\alpha y+\alpha z)}-1]^q\left(1+\frac{\alpha y z}{1+\alpha y+\alpha z}\right)^q
\\& = (1+\alpha y+\alpha z+ \alpha yz)\sum_{q \geq 0}\sum_{l,j=0}^q (-1)^{q-j}e^{jt}\alpha^l\binom{q}{j}\binom{q}{l} \frac{(yz)^{q+l}}{[(1+y)(1+z)]^q} \frac{ e^{jt\alpha(y +z)}}{(1+\alpha y+\alpha z)^l}.
\end{align*}
\end{rem}

\subsection{A Combinatorial approach}
It would be interesting to prove the previous corollary by making use of the moment formula for alternating products of free random variables. Indeed, we can expand:
\begin{align}\label{MCF}
R_{m,n}^{(\alpha)}(t) &= e^{(n+m)t/2}\sum_{\pi\in NC(n+m)}\tau_{K(\pi)}[P,...,P] k_{\pi}[\underbrace{Y_t, \dots ,Y_t}_m, \underbrace{Y_t^{\star}, \dots Y_t^{\star}}_n] \nonumber
\\&= e^{(n+m)t/2} \sum_{\pi\in NC(n+m)}\alpha^{n+m+1-|\pi|}  k_{\pi}[Y_t, \dots ,Y_t, Y_t^{\star}, \dots Y_t^{\star}],
\end{align}
and use the explicit expression of the mixed-cumulants:
\begin{equation*}
k_{j+s}[\underbrace{Y_t, \dots ,Y_t}_j, \underbrace{Y_t^{\star}, \dots Y_t^{\star}}_s], \quad j,s \geq 0,
\end{equation*}
proved in \cite{DGN}, Theorem 4.4. However, the computations are tedious even for small values of $n$ (see below) and we shall only consider the lowest-order mixed moment corresponding to $m\geq 1,  n=1,$ for which we prove: 
\begin{pro}
For any $m \geq 1$, 
\begin{equation*}
R_{m,1}^{(\alpha)}(t) =  e^{t}\sum_{r=1}^{m+1}(-\alpha e^{t/2})^{r}(v_{r-1}(t)-e^{-t}v_r(t)) \sum_{\substack{i_1+\ldots +i_r= m+1-r \\ i_1 \geq 0, \dots, i_r \geq 0}}Q_{i_1}(\alpha t)\ldots Q_{i_r}(\alpha t).
\end{equation*}
where 
\begin{equation*}
v_r(t) :=\sum_{j=0}^{r-2}\binom{r-2}{j}\frac{(r-1)^jt^j}{(r-1)\ldots (r-j)}, r \geq 2, \quad v_1(t) \equiv 1, v_0(t) \equiv 0.
\end{equation*}
Moreover, the inner sum may be further expressed as a weighted sum of Laguerre polynomials: for any $1 \leq r \leq m$: 
\begin{equation*}
\sum_{\substack{i_1+\ldots +i_r= m+1-r \\ i_1 \geq 0, \dots, i_r \geq 0}} Q_{i_1}(\alpha t)\ldots Q_{i_r}(\alpha t) = \frac{1}{m+1-r} \sum_{j=1}^{r \wedge (m+1-r)} j \binom{r}{j}L_{m+1-r-j}^{(j)}(\alpha (m+1-r)t),
\end{equation*} 
while it obviously reduces to one when $r= m+1$. 
\end{pro}

\begin{proof}
Specialize the moment formula \eqref{MCF} to $n=1$ and fix there the block $V_0$ of $\pi$ which contains the element $1$. Then, the trace property allows to rewrite $R_{m,1}^{(\alpha)}(t)$ as:
\begin{align*}
R_{m,1}^{(\alpha)}(t)&=e^{(m+1)t/2} \sum_{r=1}^{m+1}\sum_{|V_0|=r}\sum_{\pi=(V_0,...)\in NC(m+1)}\alpha^{m+2-|\pi|} k_{\pi}[Y_t^{\star},Y_t,...,Y_t].
\end{align*}
Moreover, if $V_0=(1=s_1,s_2,...,s_r)$ then $\pi$ decomposes as $V_0\cup \pi_1\cup...\cup \pi_r$ where $\pi_i\in NC(s_{i}+1,s_{i}+2,...,s_{i+1}-1)$ with the convention $s_{r+1} := m+2$. 
Hence, the cumulant functional $k_{\pi}$ factorizes as
\begin{equation*}
k_{\pi} = k_{r}[Y_t^{\star},\underbrace{Y_t,...,Y_t}_{r-1}]k_{\pi_1}(Y_t)...k_{\pi_r}(Y_t)
\end{equation*} 
and (see \cite[Remark 4.5]{DGN}):
\begin{equation*}
 k_{r}[Y_t^{\star},\underbrace{Y_t,...,Y_t}_{r-1}] =(-e^{-t/2})^{r-2}(v_{r-1}(t)-e^{-t}v_r(t)), \quad r \geq 1.
\end{equation*}

As a result,
\begin{multline*}
R_{m,1}^{(\alpha)}(t)=  e^{(m+1)t/2}\sum_{r=1}^{m+1}(-e^{-t/2})^{r-2}(v_{r-1}(t)-e^{-t}v_r(t))
 \sum_{\substack{\pi=V_0\cup \pi_1\cup...\cup \pi_r\in NC(m+1) \\ V_0=(1,s_2,...,s_r)}}\alpha^{m+2-|\pi|}k_{\pi_1}...k_{\pi_r}.
\end{multline*} 
 But if $i_k=s_{k+1}-s_k-1\in\{0,\ldots,m+1-r\}, 1 \leq k \leq r$, then $i_1+\ldots +i_r+r=m+1$ and 
\begin{align*}
R_{m,1}^{(\alpha)}(t)=& e^{(m+1)t/2}\sum_{r=1}^{m+1}(-e^{-t/2})^{r-2}(v_{r-1}(t)-e^{-t}v_r(t))
\\& \times \sum_{1<s_2<...<s_r\leq m+1}\left(\sum_{\pi_1\in NC(2,...,s_{2}-1)}\alpha^{i_1+1-|\pi_1|}k_{\pi_1}\right)\ldots\left(\sum_{\pi_r\in NC(s_r+1,...,m+1)}\alpha^{i_r+1-|\pi_r|}k_{\pi_r}\right)
\\& = e^{(m+1)t/2})\sum_{r=1}^{m+1}(-e^{-t/2})^{r-2}(v_{r-1}(t)-e^{-t}v_r(t))
 \sum_{i_1+\ldots +i_r+r= m+1}\tau[(PY_t)^{i_1}]\ldots\tau[(PY_t)^{i_r}],
\\& = e^{t}\sum_{r=1}^{m+1}(-\alpha e^{t/2})^{r}(v_{r-1}(t)-e^{-t}v_r(t)) \sum_{i_1+\ldots +i_r= m+1-r}Q_{i_1}(\alpha t)\ldots Q_{i_r}(\alpha t).
\end{align*}
Finally, write
\begin{equation*}
\sum_{i_1+\ldots +i_r= m+1-r}Q_{i_1}(\alpha t)\ldots Q_{i_r}(\alpha t) =\sum_{j_1+\ldots +j_r= m+1}Q_{j_1-1}(\alpha t)\ldots Q_{j_r-1}(\alpha t)
\end{equation*} 
and consider the series
\begin{equation*}
U(t,z):=\sum_{k\geq1}Q_{k-1}(t)z^k = \sum_{k\geq0}Q_{k}(t)z^{k+1}=z(1+\eta(t,z)).
\end{equation*}
On the one hand, we have
\begin{align*}
(U(t,z))^r &=\sum_{k\geq r}\left(\sum_{j_1+\ldots +j_r= k}Q_{j_1-1}(\alpha t)\ldots Q_{j_r-1}(\alpha t)\right)z^k.
\end{align*}
On the other hand, Lemma 4.2 in \cite{Dem} implies that:
\begin{align*}
(U(t,z))^r&=z^r(1+\eta(t,z))^r = z^r\sum_{j=0}^r\binom{r}{j}(\eta(t,z))^j
\\&= z^r + \sum_{j=1}^r j \binom{r}{j}\sum_{k\geq j}\frac{1}{k}L^{(j)}_{k-j}(kt)z^{k+r}
\\&= z^r + \sum_{k\geq 1}\sum_{j=1}^{r \wedge k} j\binom{r}{j}\frac{1}{k}L^{(j)}_{k-j}(kt)z^{k+r}
\\&=z^r + \sum_{k\geq r+1}\frac{1}{k-r}\sum_{j=1}^{r \wedge (k-r)}j\binom{r}{j}L^{(j)}_{k-r-j}((k-r)t)z^{k}.
\end{align*}
Equating the Taylor coefficients of both expansions of $(U(t,z))^r$ and substituting $t \rightarrow \alpha t$, we are done.  
\end{proof}

\begin{rem}
Adapting the previous proof to $R_{m,2}^{(\alpha)}(t)$ and keeping the same notations, we get: 
\begin{align}\label{sum}
R_{m,2}^{(\alpha)}(t)=&e^{(m+2)t/2}\sum_{r=1}^{m+2}k_r
\sum_{1<s_2<...<s_r\leq m+2}m_{i_1}\ldots m_{i_r},
\end{align}
where
\begin{equation*}
 m_{i_l}=\left(\sum_{\pi_l\in NC(s_l+1,...,s_{l+1}-1)}\alpha^{i_l+1-|\pi_l|}k_{\pi_l}\right),
\end{equation*}
and $k_r$ is the mixed cumulant corresponding to the block $V_0$.
Decomposing further the sum \eqref{sum} according to $s_2 = 2$ or $s_2 > 2$), we arrive at:
 \begin{align*}
R_{m,2}^{(\alpha)}(t)=&e^{(m+2)t/2}\sum_{r=2}^{m+2}k_{r}[Y_t^{\star}, Y_t^{\star}, \underbrace{Y_t,...,Y_t}_{r-2}] \sum_{2<s_3<...<s_r\leq m+2}m_{i_2}\ldots m_{i_r}
 \\&+e^{(m+2)t/2}\sum_{r=1}^{m+2}k_{r}[Y_t^{\star}, \underbrace{Y_t,...,Y_t}_{r-1}]  \sum_{2<s_2<...<s_r\leq m+2}m_{i_1}\ldots m_{i_r}
\\=&e^{(m+2)t/2}\sum_{r=2}^{m+2}k_{r}[Y_t^{\star}, Y_t^{\star}, \underbrace{Y_t,...,Y_t}_{r-2}] \sum_{i_2+\ldots +i_r+r= m+2}\tau[(PY_t)^{i_2}]\ldots\tau[(PY_t)^{i_r}]
\\&+e^{(m+2)t/2}\sum_{r=1}^{m+2}k_{r}[Y_t^{\star}, \underbrace{Y_t,...,Y_t}_{r-1}] \sum_{i_1+\ldots +i_r+r= m+2}\tau[PY_t^*(PY_t)^{i_1-1}]\tau[(PY_t)^{i_2}]\ldots\tau[(PY_t)^{i_r}].
\end{align*}
All the terms appearing in the last equality admit explicit expressions but they clearly lead to a cumbersome expression of $R_{m,2}^{(\alpha)}(t)$. A similar, yet more complicated, formula may be derived along the same lines for $R_{m,n}^{(\alpha)}(t)$.
 \end{rem}

\section{Odd alternating moments} 
Recall that the even alternating moments of $PY_tP$: 
\begin{equation*}
r_0^{(\alpha)}(t) := \alpha, \quad r_n^{(\alpha)}(t) := \tau((PY_tPY_t^{\star})^n), \qquad n \geq 1, 
\end{equation*}
coincide up to a normalization with the moments of the free Jacobi process associated with a single projection, which were determined in \cite{Dem-Ham} (see Corollary 3 there). 
Recall also from \cite{DHH} that the sequence $(r_n^{(\alpha)}(t))_{n \geq 1}$ satisfies (see the proof of Proposition 1 there):  
\begin{equation}\label{RR1}
\frac{dr_n^{(\alpha)}}{dt}(t)  = - n r_{n}^{(\alpha)}(t) + n\alpha r_{n-1}^{(\alpha)}(t) + n \sum_{q=0}^{n-2} r_{n-q-1}^{(\alpha)}(t) (r_{q}^{(\alpha)}(t) - r_{q+1}^{(\alpha)}(t)) \quad n \geq 1, 
\end{equation}
where an empty sum is zero. 
In this section, we shall consider their odd counterparts: 
\begin{equation*}
s_{n,1}^{(\alpha)}(t) = \tau(\underbrace{PY_tPY_t^{\star} \cdots PY_t}_{2(2n+1) \,\textrm{terms}}) = e^{-(2n+1)t/2}\tau(\underbrace{A_tB_t\cdots A_tB_tA_t}_{2n+1 \, \textrm{terms}}), \quad n \geq 2.
\end{equation*}
In particular, we readily derive
\begin{eqnarray*}
r_1^{(\alpha)}(t) & = & \alpha^2 + \alpha(1-\alpha) e^{-t}, \\
r_2^{(\alpha)}(t) &= & \alpha(1-\alpha)[(1-3\alpha+3\alpha^2) - 2\alpha(1-\alpha)t] e^{-2t}+ 4\alpha^2(1-\alpha)^2e^{-t} + \alpha^3(2-\alpha),
\end{eqnarray*}
whence we get:
\begin{eqnarray*}
e^{t/2} s_{0,1}(t) & = &\alpha  \\ 
e^{t/2}s_{1,1}^{(\alpha)}(t) & = & \alpha[\alpha(2-\alpha)+ (1-\alpha)(1-\alpha - \alpha t) e^{-t}], \\
e^{t/2}s_{2,1}^{(\alpha)}(t) &= & \alpha \left\{\alpha^2(5-6\alpha+2\alpha^2) +\alpha(1-\alpha)((1-\alpha) (6-7 \alpha) -(4-3\alpha) \alpha t)e^{-t}\right. \\ 
 & + &\left. (1-\alpha)[(1-\alpha) \left(5 \alpha^2-4 \alpha+1\right)- (4-9\alpha+6\alpha^2)\alpha t + \frac{5}{2}(1-\alpha)(\alpha t)^2]e^{-2t}\right\}.
\end{eqnarray*}

In the following proposition, we derive the analogue of \eqref{RR1} for $s_{n,1}^{(\alpha)}(t)$: 
\begin{pro}\label{ODE}
For any $n \geq 1$, 
\begin{multline*}
\frac{ds_{n,1}^{(\alpha)}}{dt}(t)  = -\frac{2n+1}{2}s_{n,1}^{(\alpha)}(t) - \sum_{q=1}^n(2n-2q+1) s_{n-q,1}^{(\alpha)}(t)r_{q}^{(\alpha)}(t)
\\ + \sum_{q=1}^ n (2n-2q+2) s_{n-q,1}^{(\alpha)}(t)r_{q-1}^{(\alpha)}(t).
\end{multline*}
\end{pro}
\begin{proof}
From \cite{BenL}, Theorem 3.3, we infer that for any collection of continuous free It\^o processes: 
\begin{equation*}
dE_j(t) = E_{j,1}(t)dX_tE_{j,2}(t), \quad 1 \leq j \leq 2n+1, 
\end{equation*}
with respect to the same free additive Brownian motion $X$, the following holds: 
\begin{multline*}
\frac{d}{dt} \tau(E_1\cdots E_{2n+1})(t) = \sum_{1 \leq k < l \leq 2n+1} \tau(E_1(t)\cdots E_{k-1}(t)E_{k,1}(t)E_{l, 2}(t)E_{l+1}(t)\cdots E_{2n+1}(t)) \\ 
\tau(E_{k,2}(t)E_{k+1}(t) \cdots E_{l-1}(t)E_{l, 1}(t)). 
\end{multline*}
We shall specialize this formula to $E_j (t)= A_t = e^{t/2}PY_t$ when $j \geq 1$ is odd and $E_j(t) = B_t = e^{t/2}PY_t^{\star}$ otherwise, and use the free SDEs:
\begin{equation*}
dA_t = iA_tdX_t, \quad dB_t = -ie^{t/2}PdX_tY_t^{\star}. 
\end{equation*} 
As a matter of fact, we need to take into account the parity of $k,l$. Using the trace property of $\tau$, we readily get the following contributions:  
\begin{itemize}
\item $k,l$ have the same parity: 
\begin{equation*}
-\tau(\underbrace{A_tB_t\cdots A_tB_tA_t}_{2n-l+k +1\, \textrm{terms}})\tau(\underbrace{A_tB_t\cdots A_tB_t}_{l-k \, \textrm{terms}}). 
\end{equation*}
\item $k$ and $l$ do not have the same parity:  
\begin{equation*}
e^t\tau(\underbrace{A_tB_t\cdots A_tB_tA_t}_{2n-l+k\, \textrm{terms}})\tau(\underbrace{A_tB_t\cdots A_tB_t}_{l-k-1 \, \textrm{terms}}),
\end{equation*}
where an empty product equals $\tau(P) = \alpha$. 
\end{itemize}
If $(k,l)$ have the same parity then we set $l-k:=2q, 1\leq  q \leq n$ and for each fixed $q$, there are $(2n-2q+1)$ couples $(1 \leq k < l \leq 2n+1)$ such that $l-k = 2q$. Otherwise, we set $l - k = 2q-1, 1 \leq q \leq n$ and for fixed $q$ there are $(2n-2q+2)$ couples 
$(1 \leq k < l \leq 2n+1)$ such that $l-k = 2q-1$. Since 
\begin{equation*}
\tau(\underbrace{A_tB_t\cdots A_tB_tA_t}_{2n+1 \, \textrm{terms}}) = e^{(2n+1)t/2}s_{n,1}^{(\alpha)}(t), 
\end{equation*}
then 
\begin{equation*}
\frac{d}{dt}\tau(\underbrace{A_tB_t\cdots A_tB_tA_t}_{2n+1 \, \textrm{terms}}) = e^{(2n+1)t/2}\left[\frac{2n+1}{2}s_{n,1}^{(\alpha)}(t) + \frac{ds_{n,1}^{(\alpha)}}{dt}(t)\right],
\end{equation*}
and the proposition follows. 
\end{proof}
\begin{rem}
The previous proposition reduces for $n=1$ to:
\begin{align*}
\frac{ds_{1,1}^{(\alpha)}}{dt}(t) & = -\frac{3}{2}s_{1,1}^{(\alpha)}(t) - s_{0,1}^{(\alpha)}(t)r_1^{(\alpha)}(t) + 2s_{0,1}^{(\alpha)}(t)r_0^{(\alpha)}(t).
\end{align*}
Identifying $r_0^{(\alpha)}$ with $R_{0,0}^{(\alpha)}$ and since $R_{1,1}^{(\alpha)}(t) = e^t r_1^{(\alpha)}(t)$ and $s_{0,1}^{(\alpha)}(t) = \tau(PY_t) = \alpha e^{-t/2}$ then this ODE is in agreement with the one proved in Proposition \ref{ode} and satisfied by $R_{2,1}^{(\alpha)}(t) = e^{3t/2}s_{1,1}^{(\alpha)}(t)$. 
\end{rem}

Let 
\begin{equation*}
S_t^{(\alpha)}(z) := \sum_{n \geq 1}s_{n,1}^{(\alpha)}(t)z^n, \quad M_t^{(\alpha)}(z) := \sum_{n \geq 1}r_{n}^{(\alpha)}(t)z^n.
\end{equation*}
Then both series converge absolutely in the open unit disc and we readily deduce from Proposition \ref{ODE} the following linear pde:
\begin{cor}
The function $(t,z) \mapsto S_t^{(\alpha)}(z)$ satisfies the pde: 
\begin{multline*}
\partial_tS_t^{(\alpha)} = s_{0,1}(t)[2z(\alpha+M_t^{(\alpha)}) - M_t^{(\alpha)}] \\ + \left[ 2z(\alpha+M_t^{(\alpha)}) - M_t^{(\alpha)} - \frac{1}{2}\right]S_t^{(\alpha)} + [2z(\alpha+M_t^{(\alpha)})-(1+2M_t^{(\alpha)})]z\partial_zS_t^{(\alpha)}.
\end{multline*}
\end{cor} 
This pde may be written in a simpler form using elementary transformations. Indeed, straightforward computations shows that:
\begin{equation*}
V_t^{(\alpha)} (z) = \alpha + e^{t/2}S_t^{(\alpha)}(z), 
\end{equation*}
satisfies: 
\begin{equation}\label{PDE0}
\partial_tV_t^{(\alpha)} = \left[ 2z(\alpha+M_t^{(\alpha)}) - M_t^{(\alpha)}\right]V_t^{(\alpha)} + [2z(\alpha+M_t^{(\alpha)})-(1+2M_t^{(\alpha)})]z\partial_zV_t^{(\alpha)},
\end{equation}
with the initial condition: 
\begin{equation*}
V_0^{(\alpha)} (z) =   \frac{\alpha}{1-z}.
\end{equation*}
Note also that the `stationary' solution $V_{\infty}^{(\alpha)}$ corresponding to the previous pde satisfies: 
\begin{equation}\label{ODESTA}
\left[(2z(\alpha+M_{\infty}^{(\alpha)})- M_{\infty}^{(\alpha)}\right]V_{\infty}^{(\alpha)} + [2z(\alpha+M_{\infty}^{(\alpha)})-(1+2M_{\infty}^{(\alpha)})]z\partial_zV_{\infty}^{(\alpha)} = 0,
\end{equation}
where 
\begin{equation*}
 M_{\infty}^{(\alpha)}(z) = \lim_{t \rightarrow \infty}M_{t}^{(\alpha)}(z) =  \sum_{n \geq 1} \tau((PUPU^{\star})^n)z^n,
\end{equation*}
and $U$ is Haar unitary random variable. The explicit expression of the generating function $M_{\infty}^{(\alpha)}$ may be obtained using standard techniques from free probability (for instance compression by a free projection, \cite{Nic-Spe}) and is given by:
\begin{equation*}
 M_{\infty}^{(\alpha)}(z) = \frac{2\alpha-1+\sqrt{1-4\alpha(1-\alpha)z}}{\sqrt{1-z}} - \alpha, \quad |z| < 1.
\end{equation*}
Due to the complicated form of the pde satisfied by $V^{(\alpha)}$, we shall focus in the next paragraph on the particular case corresponding to $\alpha=1/2$. The restriction to this value is already present in the study of the free Jacobi process (\cite{DHH}, \cite{Ham}) and the simplicity of the computations in this case is justified at the operator-algebraic level by the fact that the symmetry $2P-{\bf 1}$ associated with $P$ has trace zero. At the analytical level, the method of characteristics allows to solve the pde in this case. 
\subsection{The case $\alpha = 1/2$}
Let 
\begin{equation*}
W_t^{(\alpha)}(z) := (1-z)V_t^{(\alpha)}, \quad W_0^{(\alpha)}(z) = \alpha,
\end{equation*}
and 
\begin{equation*}
N_t^{(\alpha)}:= \alpha + M_t^{(\alpha)}. 
\end{equation*}
Then, 
\begin{teo}
The following equality holds 
\begin{equation*}
\left[W_t^{(1/2)}(z)\right]^2 =  \frac{e^{t}(1-z)}{4z}\left[4(1-z)\left(N_t^{(1/2)}(z)\right)^2 - 1\right].
\end{equation*}
wherever both sides are analytic.  
\end{teo}

\begin{proof}
Before proving this equality, let us explain, without much care about technical details and using the method of characteristics, the origin of its RHS. First of all, straightforward computations yield: 
\begin{equation*}
\partial_tW_t^{(1/2)} = -M_t^{(1/2)} W_t^{(1/2)} + z(z-1)(1+2M_t^{(1/2)})\partial_zW_t^{(1/2)},
\end{equation*}
or equivalently,
\begin{equation}\label{PDE1}
\partial_tW_t^{(1/2)} = \left(\frac{1}{2}-N_t^{(1/2)}\right) W_t^{(1/2)} + 2z(z-1)N_t^{(1/2)}\partial_zW_t^{(1/2)}. 
\end{equation}
Next, recall from \cite{DHH} that: 
\begin{align}\label{PDE2}
\partial_tN_t^{(1/2)} &=  z\partial_z\left[(z-1)(N_t^{(1/2)})^2\right] \nonumber
\\& =  z(N_t^{(1/2)})^2 + 2z(z-1) (N_t^{(1/2)}) \partial_z(N_t^{(1/2)}).
\end{align}
It follows that the characteristics of both pdes \eqref{PDE1} and \eqref{PDE2} satisfy locally the same ODE: 
\begin{equation}\label{ODECH}
z'(t) + 2z(z-1)N_t^{(1/2)}(z(t)) = 0, \quad z(0) := z_0, 
\end{equation}
therefore 
\begin{eqnarray*}
\frac{d}{dt}\left[W_t^{(1/2)}(z(t))\right] & = & \left(\frac{1}{2}-N_t^{(1/2)}(z(t))\right) W_t^{(1/2)}(z(t)) \\ 
& = &  \frac{1}{2}\left(1- \frac{z'(t)}{z(t)(1-z(t))}\right) W_t^{(1/2)}(z(t))
\end{eqnarray*}
is completely integrable. Actually, we have locally:  
\begin{equation*}
[W_t(z(t))]^2 = \frac{e^{t}}{4}\frac{(1-z(t))z_0}{z(t)(1-z_0)}, 
\end{equation*}
where this expression is understood as $[W_t(z(t))]^2 = 1/4$ when $z_0 = 0$ since then $z(t) = 0$. Now, recall from \cite{DHH} the expression: 
\begin{equation*}
N_t^{(1/2)}(z) = \frac{1}{2\sqrt{1-z}} \left[1+2\eta\left(2t,e^{-t}\frac{z}{(1+\sqrt{1-z})^2}\right)\right], \quad |z| <1, 
\end{equation*}
and introduce:
\begin{equation*}
y(t) = e^{-t} \lambda(z(t)), \quad \lambda(z) := \frac{z}{(1+\sqrt{1-z})^2} = \frac{1-\sqrt{1-z}}{1+\sqrt{1-z}}.
\end{equation*}
Then $\lambda$ is a biholomorphic map from the open unit disc to the slit plane $\mathbb{C} \setminus [1,\infty[$ satisfying:
\begin{equation*}
\lambda'(z) = \frac{\lambda(z)}{z\sqrt{1-z}}, 
\end{equation*}
whence the characteristic equation \eqref{ODECH} reduces to:
\begin{equation*}
y'(t) = 2y(t) \eta\left(2t, y(t)\right), \quad y(0) := y_0 = \lambda(z_0). 
\end{equation*}
Up to the time change $t \mapsto t/2$, this is the ODE satisfied by the characteristic curves of \eqref{PDEUBM} (though we do not know whether uniqueness holds for this ODE) along which 
\begin{equation*}
\eta(2t, y(t)) = \eta(0,y_0) = \frac{y_0}{1-y_0} \quad \Leftrightarrow \quad y_0 = \frac{\eta(2t, y(t))}{1+\eta(2t, y(t))} = y(t)e^{-2t\eta(2t, y(t))},
\end{equation*}
where the last equality follows from the functional relation (see e.g. \cite{Biane1}): 
\begin{equation}\label{FR}
\frac{\eta(2t, w)}{1+\eta(2t, w)} e^{2t\eta(2t, w)} = w , \quad |w| \leq e^{-t}. 
\end{equation}
In particular, we get: 
\begin{equation*}
y(t) = y_0 e^{2ty_0/(1-y_0)},  \quad \Rightarrow \quad z(t) = \lambda^{-1}\left(y_0 e^{t(1+y_0)/(1-y_0)}\right), 
\end{equation*}
and
\begin{align*}
z_0 = \lambda^{-1} (y_0) = \lambda^{-1}\left(\frac{\eta(2t, y(t))}{1+\eta(2t, y(t))}\right) = 1- \frac{1}{[1+2\eta(2t, e^{-t}\lambda(z(t)))]^2}.
\end{align*}
Consequently, 
\begin{equation*}
[W_t(z(t))]^2  =   \frac{e^{t}(1-z(t))[[1+2\eta(2t, e^{-t}\lambda(z(t)))]^2 - 1]}{4z(t)}
\end{equation*}
which suggests the expression of $[W_t(z(t))]^2$. Coming back to the proof of the theorem, it suffices to prove that the map 
\begin{equation*}
(t,z) \mapsto \frac{e^{t}(1-z)}{4z}\left[4(1-z)\left(N_t^{(1/2)}(z)\right)^2 - 1\right],
\end{equation*}
which is analytic $\mathbb{C} \setminus [1,\infty[$, solves the Cauchy problem: 
\begin{eqnarray*}
\partial_t(W_t^{(1/2)})^2 & = &(1-2N_t^{(1/2)})(W_t^{(1/2)})^2 + 2z(z-1)N_t^{(1/2)}\partial_z(W_t^{(1/2)})^2, \\ 
  (W_0^{(1/2)})^2(z) & = & \frac{1}{4},
\end{eqnarray*}
in a neighborhood of the origin. This is directly checked using the pde \eqref{PDE2} and the initial value: 
\begin{equation*}
4(1-z)\left(N_0^{(1/2)}(z)\right)^2 - 1 = [1+2\rho(0,\lambda(z))]^2 - 1 = \left(\frac{1+\lambda(z)}{1-\lambda(z)}\right)^2 - 1 = \frac{z}{1-z}.
\end{equation*}
\end{proof}

Using the expression of $N_t^{(1/2)}$, we can rewrite: 
\begin{align*}
\left[W_t^{(1/2)}(z)\right]^2 & = \frac{e^{t}(1-z)}{z}\left(\eta(2t,e^{-t}\lambda(z))\right)[1+ \eta(2t,e^{-t}\lambda(z))]
\end{align*}
or equivalently, 
\begin{align}\label{Flow}
\left[V_t^{(1/2)}(z)\right]^2 & = \frac{e^{t}}{z(1-z)}\left(\eta(2t,e^{-t}\lambda(z))\right)[1+ \eta(2t,e^{-t}\lambda(z))].
\end{align}
In particular, the power expansion of $\eta(2t, z)$ shows that:
\begin{equation*}
\left[V_{\infty}^{(1/2)}(z)\right]^2 = \lim_{t \rightarrow \infty} \left[V_t^{(1/2)}(z)\right]^2 = \frac{\lambda(z)}{z(1-z)} = \frac{1}{(1-z)(1+\sqrt{1-z})^2}, 
\end{equation*}
whence 
\begin{equation*}
V_{\infty}^{(1/2)}(z)  = \frac{1}{\sqrt{1-z}(1+\sqrt{1-z})}. 
\end{equation*}
Since 
\begin{equation*}
1+ 2M_{\infty}^{(1/2)}(z) = \frac{1}{\sqrt{1-z}} \quad \Leftrightarrow \quad  M_{\infty}^{(1/2)}(z) = \frac{z}{2}V_{\infty}^{(1/2)}(z)
\end{equation*}
then the expression of $V_{\infty}^{(1/2)}(z)$ displayed above is easily seen to satisfy \eqref{ODESTA} in the particular case $\alpha = 1/2$. 

On the other hand, we can use \eqref{FR} to express \eqref{Flow} as: 
\begin{align*}
\left[V_t^{(1/2)}(z)\right]^2 & = \frac{\lambda(z)}{z(1-z)}[1+ \eta(2t,e^{-t}\lambda(z))]^2e^{-2t \eta(2t,e^{-t}\lambda(z))}
\end{align*}
whence we infer that (at least locally around $z =0$):
\begin{align*}
V_t^{(1/2)}(z) & =  \frac{1}{\sqrt{1-z}(1+\sqrt{1-z})} [1+ \eta(2t,e^{-t}\lambda(z))]e^{-t \eta(2t,e^{-t}\lambda(z))}.
\end{align*}
With the help of this expression, we arrive at: 
\begin{cor}
Let 
\begin{eqnarray*}
c_0(t) &= & \frac{1}{2} \\ 
c_1(t) & = & \frac{1+(1-2t) e^{-t}}{2} \\ 
c_2(t) & = & \frac{(1-2t)e^{-t} - 2te^{-2t}}{2} \\ 
c_j(t) &=& -t e^{-jt} \sum_{k = 1}^{j-1}\frac{1}{k(j-k)}L_{k-1}^{(1)} (2kt)L_{j-k-1}^{(1)} (2(j-k+1)t) \\ 
 & - & t   e^{-(j-1)t} \sum_{k = 1}^{j-2}\frac{1}{k(j-k-1)}L_{k-1}^{(1)} (2kt)L_{j-k-2}^{(1)} (2(j-k)t), \quad j \geq 3. 
\end{eqnarray*}
Then 
\begin{equation*}
s_{n,1}(t) = \frac{e^{-t/2}}{4^n} \sum_{j=0}^n \binom{2n}{n-j}c_j(t). 
\end{equation*}
\end{cor} 
\begin{proof}
Write 
\begin{equation*}
\frac{1}{\sqrt{1-z}(1+\sqrt{1-z})} = \frac{1+\lambda(z)}{2} \frac{1+\lambda(z)}{1-\lambda(z)}
\end{equation*}
so that 
\begin{align*}
V_t^{(1/2)}(z) & =  \frac{1+\lambda(z)}{1-\lambda(z)} \frac{1+\lambda(z)}{2}  [1+ \eta(2t,e^{-t}\lambda(z))]e^{-t \eta(2t,e^{-t}\lambda(z))}.
\end{align*}
Now recall from \cite{Man-Sri}, p.357, the following fact: if $(c_n)_{n \geq 0}, (b_n)_{n \geq 0}$ are two sequences of real numbers related by 
\begin{equation}\label{Conv}
b_n = \sum_{k=0}^n\binom{2n}{n-k}c_k,
\end{equation}
then 
\begin{equation}\label{Rela} 
\sum_{n \geq 0}b_n\frac{z^n}{4^n} = \frac{1+\lambda(z)}{1-\lambda(z)} \sum_{n \geq 0}c_n[\lambda(z)]^n 
\end{equation}
whenever both series converges absolutely. Accordingly if 
\begin{equation*}
V_t^{(1/2)}(z) := \frac{1}{2} + e^{t/2} \sum_{n \geq 1}s_{n,1}(t) z^n = \sum_{n \geq 0}b_n(t)\frac{z^n}{4^n}
\end{equation*}
and 
\begin{equation*}
\frac{1+z}{2}e^{-2t\eta(2t,e^{-t}z)}[1+ \eta(2t,e^{-t}z)] = \sum_{n \geq 0}c_n(t)z^n,
\end{equation*}
then relation \eqref{Rela} holds for the sequences $(b_n(t))_{n \geq 0}$ and $(c_n(t))_{n \geq 0}$ with $c_0(t)= b_0(t) = 1/2$. The latter sequence may be determined using results proved in \cite{Dem}. More precisely, the proof of Proposition 6.1 there shows that 
\begin{equation*}
e^{-2t\eta(2t,e^{-t}z)} = 1 - 2t \sum_{j \geq 1}\frac{e^{-jt}}{j} L_{j-1}^{(1)} (2(j+1)t) z^j
\end{equation*}
which in turn entails: 
\begin{multline*}
e^{-2t\eta(2t,e^{-t}z)}[1+ \eta(2t,e^{-t}z)] = 1 + (1- 2t) e^{-t}z -2t  \sum_{j \geq 2} e^{-jt} z^j\\ 
\sum_{k = 1}^{j-1}\frac{1}{k(j-k)}L_{k-1}^{(1)} (2kt)L_{j-k-1}^{(1)} (2(j-k+1)t).
\end{multline*}
Keeping in mind the factor $(1+z)/2$, one then extracts the coefficients $(c_j(t))_{j \geq 0}$ and applies \eqref{Conv}.
\end{proof}
\begin{rem}
Since $s_{k,1}(\infty) = 0$ for any $k \geq 0$, then \eqref{OddEven} is equivalent when $\alpha = 1/2$ to: 
\begin{multline*}
\sum_{k=1}^n\left[s_{k-1,1}^{(1/2)}(t) - s_{k-1,1}^{1/2)}(\infty)\right]\left[s_{n-k,1}^{(1/2)}(t) - s_{n-k,1}^{1/2)}(\infty)\right] = \\ 
\sum_{k=1}^n \left[r_{k}^{(1/2)}(t) - r_{k}^{(1/2)}(\infty)\right]\left[r_{n-k}^{(1/2)}(t) + r_{n-k}^{(1/2)}(\infty)\right].
\end{multline*}
We do not know whether this equality admits or not a combinatorial interpretation. 
\end{rem}

\begin{rem} 
At the analytic level, \eqref{OddEven} also implies that 
\begin{equation*}
\frac{d}{dt} \left\{\sum_{k=1}^n s_{k-1,1}^{(1/2)}(t)s_{n-k,1}^{(1/2)}(t) - \sum_{k=0}^n r_{k}^{(1/2)}(t)r_{n-k}^{(1/2)}(t)\right\} = 0. 
\end{equation*} 
Noting that 
\begin{align*}
\frac{d}{dt} \sum_{k=1}^n s_{k-1,1}^{(1/2)}(t)s_{n-k,1}^{(1/2)}(t) & = 2 \sum_{k=1}^n \frac{ds_{k-1,1}^{(1/2)}(t)}{dt}s_{n-k,1}^{(1/2)}(t)
 \end{align*}
and using Proposition \ref{ODE}, we get: 
\begin{align*}
\frac{d}{dt} \sum_{k=1}^n s_{k-1,1}^{(1/2)}(t)s_{n-k,1}^{(1/2)}(t) & =  -\sum_{k=1}^n(2k-1)s_{k-1}(t)s_{n-k}(t) - 2 \sum_{k=1}^n\sum_{q=1}^{k-1}(2k-2q-1)s_{k-q-1}(t) s_{n-k}(t) r_q(t) 
\\& + 2\sum_{k=1}^n\sum_{q=1}^{k-1}(2k-2q)s_{k-q-1}(t)s_{n-q}(t)r_{q-1}(t).
\end{align*}
Now, 
\begin{equation*}
-\sum_{k=1}^n(2k-1)s_{k-1}(t)s_{n-k}(t) = -n\sum_{k=1}^n s_{k-1}(t)s_{n-k}(t), 
\end{equation*}
\begin{align*}
2\sum_{k=1}^n\sum_{q=1}^{k-1}(2k-2q-1)s_{k-q-1}(t) s_{n-k}(t) r_q(t) & = 2\sum_{q=1}^n r_q(t) \sum_{k=q+1}^n (2k-2q-1)s_{k-q-1}(t) s_{n-k}(t)
\\& = 2\sum_{q=1}^{n-1} r_q(t) \sum_{k=0} ^{n-q-1} (2k+1)s_{k}(t) s_{n-q-k-1}(t)
\\& =  2\sum_{q=1}^{n-1} r_{n-q}(t) \sum_{k=0}^{q-1} (2k+1)s_{k}(t) s_{q-k-1}(t)
\\& = 2 \sum_{q=1}^{n-1} q r_{n-q}(t) \sum_{k=0}^{q-1} s_{k}(t) s_{q-k-1}(t).
\end{align*}
and similarly 
\begin{align*}
2\sum_{k=1}^n\sum_{q=1}^{k-1}(2k-2q)s_{k-q-1}(t) s_{n-k}(t) r_{q-1}(t) & = 2 \sum_{q=1}^{n-1} (q+1) r_{n-q-1}(t) \sum_{k=0}^{q-1} s_{k}(t) s_{q-k-1}(t).
\end{align*}
Consequently, 
\begin{align}\label{Derivee1}
\frac{d}{dt} \sum_{k=1}^n s_{k-1,1}^{(1/2)}(t)s_{n-k,1}^{(1/2)}(t) & = -n\sum_{k=1}^n s_{k-1}(t)s_{n-k}(t) - 2 \sum_{q=1}^{n-1} q r_{n-q}(t) \sum_{k=0}^{q-1} s_{k}(t) s_{q-k-1}(t) \nonumber
\\&  + 2 \sum_{q=1}^{n-1} (q+1) r_{n-q-1}(t) \sum_{k=0}^{q-1} s_{k}(t) s_{q-k-1}(t) \nonumber 
\\& =  2 \sum_{q=2}^{n} q r_{n-q}(t)\sum_{k=0}^{q-2} s_{k}(t) s_{q-k-2}(t) - 2\sum_{q=1}^n q r_{n-q}(t) \sum_{k=0}^{q-1} s_{k}(t) s_{q-k-1}(t).
\end{align}
where the second equality follows from $r_{0}(t) = 1/2$. 

On the other hand, write \eqref{RR1} with $\alpha=1/2$ in the following form: 
\begin{align*}
\frac{dr_n^{(1/2)}}{dt}(t) & = - \frac{1}{2} n r_{n}^{(1/2)}(t) + n \sum_{k=0}^{n-1} r_{n-k-1}^{(1/2)}(t) (r_{k}^{(1/2)}(t) - r_{k+1}^{(1/2)}(t)),
\\& = - \frac{1}{2} n r_{n}^{(1/2)}(t) + n \sum_{k=1}^{n} r_{n-k}^{(1/2)}(t) r_{k-1}^{(1/2)}(t) - n \sum_{k=1}^{n} r_{n-k}^{(1/2)}(t) r_{k}^{(1/2)}(t),
\end{align*}
and use it to get: 
\begin{align*}
\frac{d}{dt} \sum_{k=0}^n r_{k}^{(1/2)}(t)r_{n-k}^{(1/2)}(t) & =  - \sum_{k=0}^nkr_{k}^{(1/2)}(t)r_{n-k}^{(1/2)}(t) - 2 \sum_{k=0}^nk \sum_{q=1}^kr_{k-q}^{(1/2)}(t)r_{q}^{(1/2)}(t)r_{n-k}^{(1/2)}(t)
\\& + 2  \sum_{k=0}^nk \sum_{q=1}^kr_{k-q}^{(1/2)}(t)r_{q-1}^{(1/2)}(t)r_{n-k}^{(1/2)}(t).
\end{align*}
Moreover, 
\begin{align*}
- \sum_{k=0}^nkr_{k}^{(1/2)}(t)r_{n-k}^{(1/2)}(t) = -\frac{n}{2} \sum_{k=0}^nr_{k}^{(1/2)}(t)r_{n-k}^{(1/2)}(t) 
\end{align*}
and similarly, 
\begin{align*}
2 \sum_{k=0}^nk \sum_{q=1}^kr_{k-q}^{(1/2)}(t)r_{q}^{(1/2)}(t)r_{n-k}^{(1/2)}(t) & = 2 \sum_{q=1}^n r_{q}^{(1/2)}(t)\sum_{k=q}^nk r_{k-q}^{(1/2)}(t) r_{n-k}^{(1/2)}(t)
\\&= 2 \sum_{q=0}^{n-1} r_{n-q}^{(1/2)}(t)\sum_{k=0}^q(k+q) r_{k}^{(1/2)}(t) r_{q-k}^{(1/2)}(t)
\\& = 3\sum_{q=0}^{n-1} qr_{n-q}^{(1/2)}(t)\sum_{k=0}^q r_{k}^{(1/2)}(t) r_{q-k}^{(1/2)}(t),
\end{align*}
\begin{align*}
2\sum_{k=0}^nk \sum_{q=1}^kr_{k-q}^{(1/2)}(t)r_{q-1}^{(1/2)}(t)r_{n-k}^{(1/2)}(t) &= 2 \sum_{q=1}^n r_{q-1}^{(1/2)}(t)\sum_{k=q}^nk r_{k-q}^{(1/2)}(t) r_{n-k}^{(1/2)}(t)
\\&= 2 \sum_{q=0}^{n-1} r_{n-q-1}^{(1/2)}(t)\sum_{k=0}^q(k+q) r_{k}^{(1/2)}(t) r_{q-k}^{(1/2)}(t)
\\& = 3\sum_{q=0}^{n-1} qr_{n-q-1}^{(1/2)}(t)\sum_{k=0}^q r_{k}^{(1/2)}(t) r_{q-k}^{(1/2)}(t).
\end{align*}
Altogether 
\begin{align}\label{Derivee2}
\frac{d}{dt} \sum_{k=0}^n r_{k}^{(1/2)}(t)r_{n-k}^{(1/2)}(t) & =  - \frac{n}{2} \sum_{k=0}^nr_{k}^{(1/2)}(t)r_{n-k}^{(1/2)}(t)  - 3\sum_{q=0}^{n-1} qr_{n-q}^{(1/2)}(t)\sum_{k=0}^q r_{k}^{(1/2)}(t) r_{q-k}^{(1/2)}(t) \nonumber
\\& +3\sum_{q=0}^{n-1} qr_{n-q-1}^{(1/2)}(t)\sum_{k=0}^q r_{k}^{(1/2)}(t) r_{q-k}^{(1/2)}(t).
\end{align}
Formula \eqref{OddEven} states that \eqref{Derivee1} and \eqref{Derivee2} differ by a constant independent of $t$. 
\end{rem}

\subsection{The general case} 
The equality \eqref{Flow} takes the form: 
\begin{equation}\label{OddEven}
e^{-t}z\left[V_t^{(1/2)}(z)\right]^2 = \left(N_t^{(1/2)}(z)\right)^2 - \left(N_{\infty}^{(1/2)}(z)\right)^2
\end{equation}
which suggests to express $\left[V_t^{(\alpha)}(z)\right]^2$ through $\left[N_t^{(\alpha)}(z)\right]^2$. However, we can prove  that \eqref{OddEven} fails in general unless $\alpha = 1/2$. 

On the other hand, the characteristic curves $t \mapsto z^{(\alpha)}(t) \equiv z(t)$ of the pde \eqref{PDE0} are solutions of the ODE: 
\begin{equation*}
z'(t) + 2z(t)[z(t)-1] N_t^{(\alpha)}(z(t)) + (2\alpha-1)z(t) = 0.
\end{equation*}
Now, 
\begin{equation*}
N_{t}^{(\alpha)}(z) = \frac{2\alpha-1}{2(1-z)}  + \frac{H_t(\lambda(z))}{2\sqrt{1-z}}, \quad |z| < 1,
\end{equation*}
where $H_t$ is the Herglotz transform of the spectral distribution of the unitary operator: 
\begin{equation*}
RY_tRY_t^{\star}, \quad R := 2P-{\bf 1}. 
\end{equation*}
Consequently, the ODE satisfied by the characteristic curves reads: 
\begin{equation*}
z'(t) = z(t)\sqrt{1-z(t)}H_t(\lambda(z(t)).
\end{equation*}
Setting $y(t) = \lambda(z(t))$ then we have (at least for small $t$): 
\begin{equation*}
y'(t) = y(t)H_t(y(t)).
\end{equation*}
This ODE reminds the flow governing the spectral dynamics of $RY_tRY_t^{\star}$ and along which the holomorphic map
\begin{equation*}
K_t(y) = \sqrt{H_t(y)^2 - (2\alpha-1)^2\left(\frac{1+y}{1-y}\right)^2}, \quad |y| < 1, 
\end{equation*}
is constant (\cite{Ham1}). 

\subsection{The combinatorial approach} 
We close the paper by discussing the combinatorial approach to the odd alternating moments of $PY_tP$ in the particular case $\alpha =1/2$. Actually, the odd cumulants $k_{2j+1}(P, \dots, P)$ vanish except $k_1(P) = \tau(P)= 1/2$, so that 
the non zero contribution in the formula
\begin{equation*}
s_{n,1}^{(\alpha)}(t) = \sum_{\pi\in NC(2n+1)}k_{\pi}[P,...,P] \tau_{K(\pi)}[Y_t, Y_t^{\star}, \dots, Y_t, Y_t^{\star},Y_t],
\end{equation*}
comes only from partitions $\pi$ which contain an odd number of singletons together with blocks of even sizes. Consequently, such partition may be decomposed as: 
\begin{equation*}
\pi = 1^{2k+1} \cup \rho 
\end{equation*}
where $0 \leq k \leq n$ and $\rho$ exhausts the set $NCE(2n-2k)$ of even non crossing partitions (\cite{Nic-Spe}). Moreover, any block of the Kreweras complement $K(\pi)$ may contain only integers with the same parity and/or consecutive integers in which case cancellations occur due to the unitarity of $Y_t$. As a matter of fact, we get: 
\begin{equation*}
s_{n,1}^{(\alpha)}(t) = \sum_{k=0}^n\frac{1}{2^{2k+1}} \sum_{\rho \in NCE(2n-2k)} k_{\rho}[P,...,P]\tau_{K(1^{2k+1} \cup \rho)}[Y_t, Y_t^{\star}, \dots, Y_t, Y_t^{\star},Y_t].
\end{equation*}
Recalling the free cumulants:
\begin{equation*}
k_1(P) = \tau(P) = \frac{1}{2}, \quad k_{2j}(P,\dots, P) = \frac{(-1)^{j-1}}{2^{2j}} C_{j-1}, j \geq 1, 
\end{equation*}
where $C_j$ is the $j$-th Catalan number (see \cite{Nic-Spe}, Exercise 11.35), we are then left with the analysis of the block structure of $K(1^{2k+1} \cup \rho)$. However, this block structure not only depends on the number of singletons but also on their positions as shown by the following example: take $n=2, k=1,$ and consider the non-crossing partitions:
\begin{equation*}
\pi_1 = \{1\}\{2\}\{3\}\{4, 5\}, \quad \pi_2 = \{1\}\{2\}\{4\}\{3, 5\}
\end{equation*}
Then,
\begin{equation*}
K(\pi_1) = \{1,2, 3, 5\}\{4\}, \quad K(\pi_2) = \{1, 2, 5\}\{3, 4\}
\end{equation*}
whose contributions are respectively: 
\begin{equation*}
\tau_{K(\pi_1)}[Y_t, Y_t^{\star}, Y_t, Y_t^{\star},Y_t] = \tau(Y_t^2)\tau(Y_t), \quad \tau_{K(\pi_2)}[Y_t, Y_t^{\star}, Y_t, Y_t^{\star},Y_t] = \tau(Y_t).
\end{equation*}
It would be interesting to pursue these computations and to obtain as simple as possible formula for $s_{n,1}^{(\alpha)}(t)$.

\end{document}